\documentclass[3p,10pt,authoryear,times]{article}

\usepackage[latin1]{inputenc}
\usepackage{amssymb}
\usepackage{graphicx}
\usepackage{amsmath}
\usepackage{amsthm}
\usepackage{mathrsfs}

\usepackage{natbib}
\setcitestyle{citesep={;}, aysep={,}}

\newtheorem{lemma}{Lemma}
\newtheorem{theorem}{Theorem}
\newtheorem{proposition}{Proposition}
\newtheorem{assumption}{Assumption}

\theoremstyle{definition}
\newtheorem{definition}[theorem]{Definition}

\newcommand{\RR}{\mathbb R}
\newcommand{\PP}{\mathbb P}
\newcommand{\ZZ}{\mathbb Z}
\newcommand{\f}[1]{\mathbf{#1}}

\begin{document}

\title{Approximation properties of isogeometric function spaces on singularly parameterized domains}

\author{
        Thomas Takacs \\
		 Department of Mathematics \\
            University of Pavia, Italy\\
            \tt{thomas.takacs@unipv.it}}

\date{\today}

\maketitle

\begin{abstract}
We study approximation error bounds of isogeometric function spaces on a specific type of singularly parameterized domains. In this context an isogeometric function is the composition of a piecewise rational function with the inverse of a piecewise rational geometry parameterization. We consider domains where one edge of the parameter domain is mapped onto one point in physical space. To be more precise, in our configuration the singular patch is derived from a  reparameterization of a regular triangular patch. On such a domain one can define an isogeometric function space fulfilling certain regularity criteria that guarantee optimal convergence. The main contribution of this paper is to prove approximation error bounds for the previously defined class of isogeometric discretizations. 
\end{abstract}

\section{Introduction}

Isogeometric analysis as presented by \cite{Hughes2005} starts from a discretization based on a B-spline or NURBS geometry parameterization. 
The standard theory, as developed in \cite{Bazilevs2006} relies on the regularity of the geometry mapping. If the geometry mapping is singular, the standard theory does not apply. Therefore, in the standard setting, the underlying geometry has to be diffeomorphic to a rectangle. Singular parameterizations have been studied in the context of isogeometric analysis in e.g. \cite{Lu2009,Cohen2010,Lipton2010}
The regularity theory has already been partially extended to singularly parameterized domains in \cite{Takacs2011,Takacs2012-1}. However, the study of approximation properties of isogeometric spaces over singular patches still remains an open problem, which we study in more detail in this paper.
In the context of finite elements there exist several related studies concerning degenerate $Q^1$ isoparametric elements, such as \cite{Jamet1977,Acosta2006}.

In Section \ref{sec:2} we present the notation and the underlying setting that we consider throughout this paper. In Section \ref{sec:spaces-operators} we introduce the hierarchical mesh as well as the corresponding function space on it. In Section \ref{subsec:operator} we define an $L^2$-stable projection operator onto the hierarchical function space. We briefly mention how the construction can be generalized to locally quasi-uniform refinements in Section \ref{subsec:generalization-quasi-uniform} and present approximation error bounds in Section \ref{sec:appr-estimates}. Finally, we conclude the paper in Section \ref{sec:conclusion}.

\section{Preliminaries}
\label{sec:2}

Isogeometric function spaces $\mathcal{V}$, as they are present in isogeometric analysis introduced in \cite{Hughes2005}, are built from an underlying B-spline or NURBS space. Hence, to introduce the notation needed, we start this preliminary section with recalling the notion of B-splines and NURBS. We do not give detailed definitions here but refer to standard literature for further reading, see \cite{PieglTiller1995,Prautzsch2002}.

\subsection{B-splines, NURBS and isogeometric functions}

Throughout this paper we consider the simple configuration of an isogeometric function space based on a single NURBS patch with uniform knots. Note that this simplification is not necessary, as we point out in Section \ref{subsec:generalization-quasi-uniform}.

For simplicity we consider uniform B-splines over the parameter domain $]0,1[$. Given a degree $p\in\ZZ^+$ and a mesh size $h = 1/2^n$, with $n \in \ZZ^+_0$, the $i$-th B-spline, for $i=1,\ldots,2^n+p$, is denoted by $b^n_{i}(s)$. Here $n$ refers to the level of (dyadic) refinement. We assume that the knot vector is open and has uniformly distributed interior knots. We moreover denote by $b^n_{i,j}(s,t)$ the product of $b^n_{i}(s)$ and $b^n_{j}(t)$. Let $\mathcal{S}_n$ be the tensor-product B-spline space spanned by $b^n_{i,j}(s,t)$. We have given a weight function $w \in \mathcal{S}_{n^*}$ on some coarse level $n^*$, with $w>0$ uniformly. For all $n\geq n^*$ we define the NURBS space as $\mathcal{N}_n = \{ \frac{v_n}{w} : v_n \in \mathcal{S}_n \}$. Moreover, let $\f G \in (\mathcal{N}_{n^*})^2$ be a geometry parameterization, with 
\begin{equation*}
\f G : \f B = \left]0,1\right[^2 \rightarrow \Omega \subset \RR^2.
\end{equation*}
For $n\geq n^*$, the h-dependent spaces $\mathcal{V}_n$ of \emph{isogeometric functions} over the open domain $\Omega = \f G (\f B)$ are defined by 
\begin{equation*}
	\mathcal{V}_n = \left\{ \varphi_n: \Omega \rightarrow \RR \; | \; 
	\varphi_n = f_n \circ \f G^{-1}, 
	\mbox{ with } f_n \in\mathcal{N}_n \right\}.
\end{equation*}
The goal of this paper is to prove approximation error bounds for h-refined isogeometric functions over singularly parameterized domains. In the following section we introduce a simple configuration of singularly parameterized domains.

\subsection{Singular tensor-product patches derived from triangular patches}

In this section we construct smooth isogeometric function spaces on singular patches, where the singular patches are derived from triangular B\'ezier patches. The configuration presented here is developed in more detail in \cite{Takacs2014}, which is based on \cite{Hu2001}.
 
Let $\f u$ be the mapping 
\begin{eqnarray*}
	\f u : \;\; ]0,1[^2 & \; \rightarrow \; & \Delta = \{ (u,v): 0< u < 1 , \; 0< v< u \} \\
	(s,t)^T & \; \mapsto \; & \left(s, s \, t\right)^T,
\end{eqnarray*}
and let $\f F: \Delta \rightarrow \RR^2$ be a regular mapping. We assume that the 
parameterization $\f G$ is given as
\begin{equation*}
	\f G = \f F \circ \f u.
\end{equation*}
All mappings $\f u$, $\f F$ and $\f G$ are defined on an open parameter domain and can be extended continuously to the boundary. 
Here, $\f G$ is singular at a part of the boundary, i.e. $\det \nabla \f G (s,t) = 0$ for all 
$(s,t) \in \{0\}\times [0,1]$. Let $\mathcal{W}_n$ be the isogeometric B-spline space on $\Delta$, i.e.
\begin{equation*}
	\mathcal{W}_n = \left\{ \varphi_n: \Delta \rightarrow \RR \; | \; 
	\varphi_n = f_n \circ \f u^{-1}, 
	\mbox{ with } f_n \in\mathcal{S}_n \right\},
\end{equation*}
where the inverse of the singular mapping $\f u$ is equal to 
\begin{equation*}
	\f u^{-1}(u,v) = \left(u,\frac{v}{u}\right)^T.
\end{equation*}
The various introduced mappings and domains are depicted in Figure \ref{figureMapping}. 
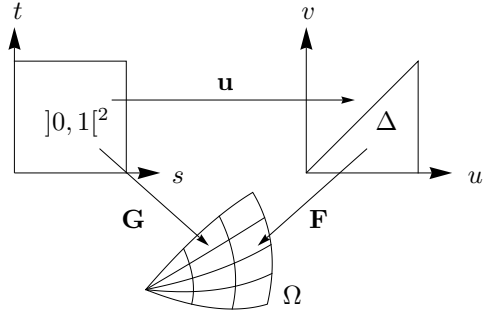
\begin{figure}[!ht]
    \centering
\begin{picture}(170,120)
\put(0,0){\includegraphics[width=0.5\textwidth]{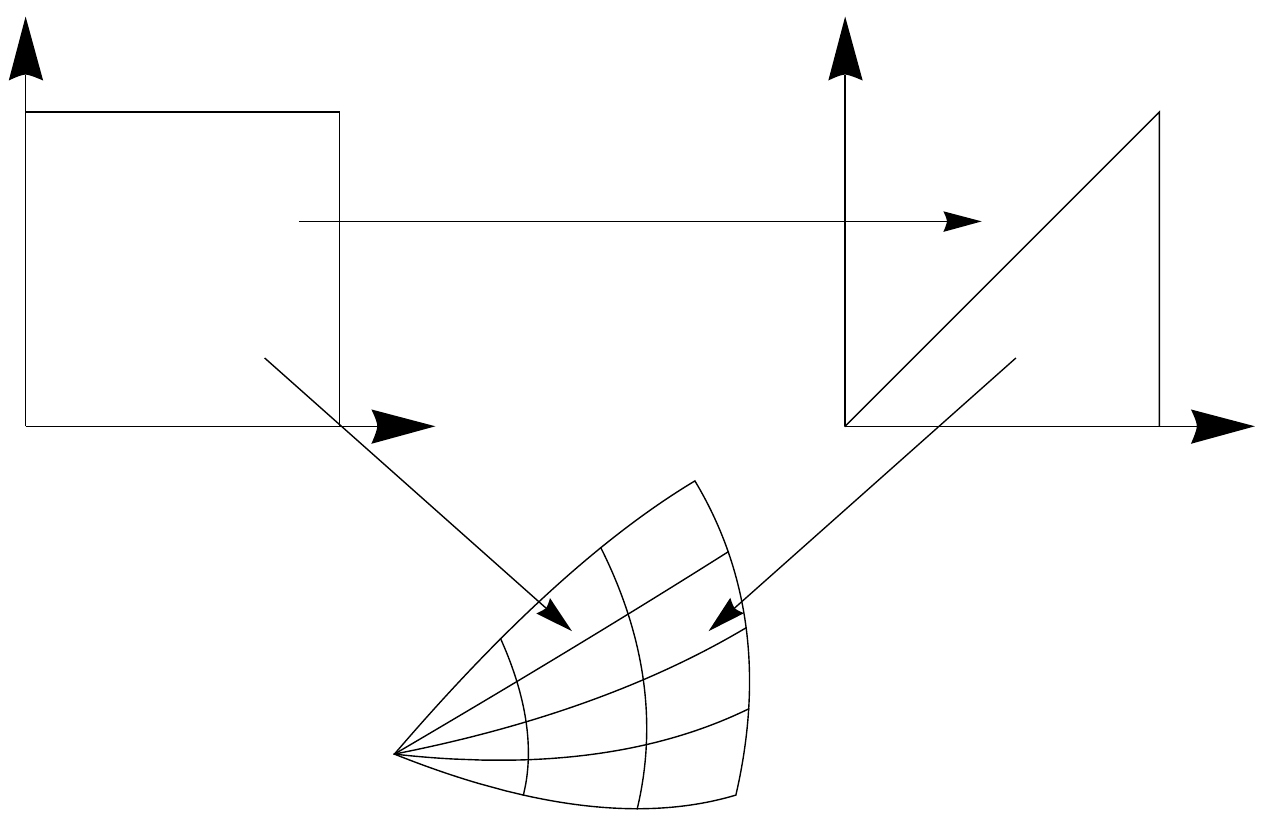}}
\put(63,50){$s$}
\put(175,50){$u$}
\put(3,112){$t$}
\put(112,112){$v$}
\put(115,33){$\f F$}
\put(44,33){$\f G$}
\put(80,85){$\f u$}
\put(105,4){$\Omega$}
\put(15,70){$]0,1[^2$}
\put(140,70){$\Delta$}
\end{picture}
    
    \caption{Mappings $\f F$, $\f u$ and $\f G$ for an example domain $\Omega$}\label{figureMapping}
\end{figure}

\subsection{Regularity conditions}

To prove approximation properties on $\Omega$, we need the following.
\begin{assumption}\label{assu:equiv-norms}
There exists a constant $C_F$, depending only on $\f F$ and on the degree $p$, such that 
\begin{equation*}
	\frac{1}{C_F} \| \varphi \|_{{H}^{k}(\Omega)} 
	\leq \| \varphi \circ \f F \|_{{H}^{k}(\Delta)} 
	\leq C_F \| \varphi \|_{{H}^{k}(\Omega)}
\end{equation*}
for all $0\leq k\leq p+1$ and for all $\varphi \in {H}^{p+1}(\Omega)$.
\end{assumption}
Considering a mapping that is continuous of order $p$ on the closure of the open domain, we have the following.
\begin{proposition}\label{prop:equiv-norms}
If $\f F \in (\mathscr{C}^{p}(\overline{\Delta}))^2$ and $\det\nabla\f F > \underline c > 0$, then Assumption \ref{assu:equiv-norms} is fulfilled.
\end{proposition}
The proof of this proposition can be carried out similarly to the proof of Lemma 3.5 in \cite{Bazilevs2006}. 
Here, the space $\mathscr{C}^{k}$ is defined as follows.
\begin{definition}
Given an open domain $D$. The space $\mathscr{C}^{k}(\overline{D})$ of $\mathscr{C}^{k}$-continuous functions on the closure of $D$ is defined as the space 
of functions $\varphi:D\rightarrow\RR$ with $\varphi\in C^k(D)$ such that there exists a unique 
limit 
\begin{equation*}
	\lim_{\substack{\f y \rightarrow \f x\\ \f y \in D}}
	\frac{\partial^{|\alpha|} \varphi(\f y)}{\partial x_1^{\alpha_1} \partial x_2^{\alpha_2}} = 
	\frac{\partial^{|\alpha|} \varphi(\f x)}{\partial x_1^{\alpha_1} \partial x_2^{\alpha_2}}
\end{equation*}
for all $\f x \in \partial D = \overline{D} \backslash D$ and for all 
$|\alpha| = \alpha_1 + \alpha_2 \leq k$. Here $C^k(D)$ is the standard space of $k$-times continuously differentiable functions on the open domain $D$. 
\end{definition}

\section{Spaces and operators on the triangle}
\label{sec:spaces-operators}

In this section we introduce a hierarchical mesh refinement and corresponding spline spaces $\widehat{\mathcal{W}}_n$ on the triangle $\Delta$. In this configuration $\widehat{\mathcal{W}}_n$ is a subspace of $\mathcal{W}_n$. Then we introduce a projection operator  $\Pi_{\widehat{\mathcal{W}}_n} : L^2(\Delta) \rightarrow \widehat{\mathcal{W}}_n$, that is locally bounded in $L^2$. 

\subsection{Hierarchical mesh and refinement}
\label{subsec:mesh}

Let $\rm{T}_n$ be the B\'ezier mesh corresponding to the function space $\mathcal{W}_n$ on $\Delta$, i.e. it is a regular mesh mapped by $\f u$. We consider a sequence of meshes $\widehat{\rm{T}}_n$, such that $\widehat{\rm{T}}_0 = \{ \Delta \}$ and 
\begin{equation*}
	\widehat{\rm{T}}_n = \{ \delta : \delta \in \rm{T}_n \cap (\Delta \backslash \Delta_{1/2}) \} \cup \{ \delta/2 : \delta \in \widehat{\rm{T}}_{n-1}\},
\end{equation*}
where 
\begin{equation*}
	\Delta_{\gamma} = \{ (u,v): 0< u < \gamma , \; 0< v< u \}
\end{equation*}
and $\delta/2 = \{(\frac{u}{2},\frac{v}{2})\,:\, (u,v)\in\delta\}$. Here, with a slight abuse of notation, we have 
\begin{equation*}
	{\rm{T}}_n \cap (\Delta \backslash \Delta_{1/2}) = 
	\{ \delta : \delta \in {\rm{T}}_{n} \mbox{ and } \delta \subseteq \Delta \backslash \Delta_{1/2} \}.
\end{equation*}
Note that all elements in $\widehat{\rm{T}}_n$ are shape regular, i.e. the radius of the largest inscribed circle of $\delta$ is of the same order $O(h)$ as the diameter of $\delta$. Figure \ref{figure:Mesh-hat-Tn} depicts the locally refined mesh for $n=1,2,3$. The meshes are defined in such a way, that the left half of the mesh $\widehat{\rm{T}}_n$ is a scaled version of the mesh $\widehat{\rm{T}}_{n-1}$. 
\begin{figure}[!ht]
    \centering
    \includegraphics[width=0.2\textwidth]{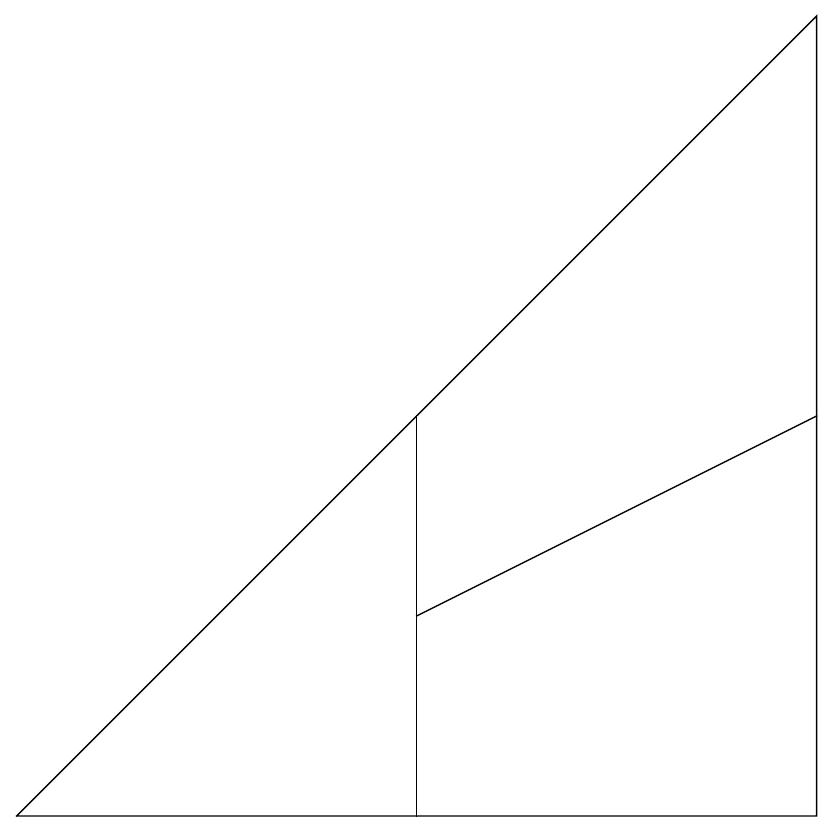}\qquad
    \includegraphics[width=0.2\textwidth]{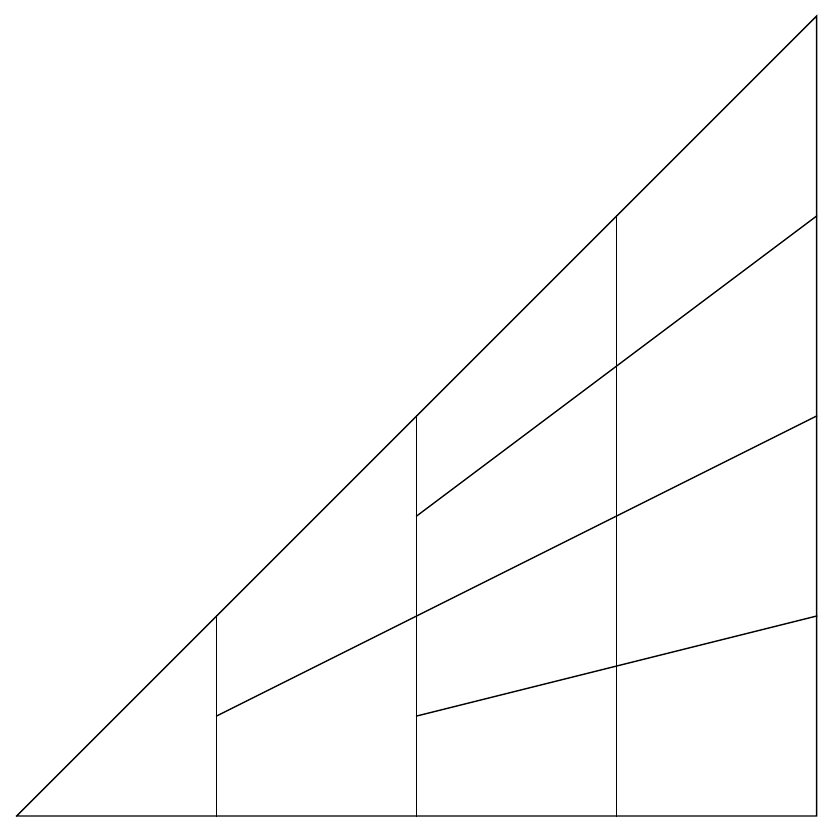}\qquad
    \includegraphics[width=0.2\textwidth]{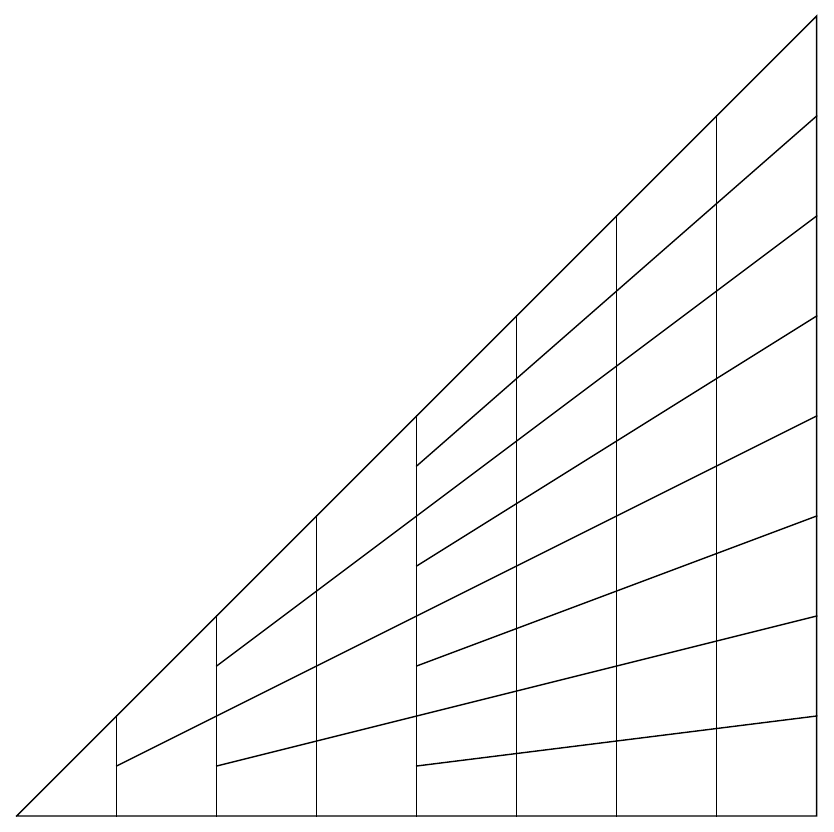}
    \caption{Meshes $\widehat{\rm{T}}_1$, $\widehat{\rm{T}}_2$ and $\widehat{\rm{T}}_3$}\label{figure:Mesh-hat-Tn}
\end{figure}
This refinement scheme is presented in more detail in \cite{Takacs2012-2}.

\subsection{Hierarchical function space}
\label{subsec:space}

Now we can define a (hierarchical) mapped B-spline space over this mesh. The following definition is based on the construction presented in \cite{Takacs2014}.
\begin{definition}
The basis $\widehat{\mathbb{S}}_n$ over $]0,1[^2$ is defined via 
\begin{equation*}
\begin{array}{rcl}
	\widehat{\mathbb{S}}_n & = & \left\{ \hat{b}^n_{i,j} = b^n_i(s) B^i_{j}(t): 1\leq i\leq p+1 \mbox{ and } 1\leq j\leq i+1 \right\} \\ 
	& \cup & \left\{ \hat{b}^n_{i,j} = b^n_i(s)b^1_j(t): p+1+1\leq i \leq p+2 \mbox{ and } 1\leq j\leq p+2 \right\} \\ 
	& \cup & \left\{ \hat{b}^n_{i,j} = b^n_i(s)b^2_j(t): p+2+1\leq i \leq p+4 \mbox{ and } 1\leq j\leq p+4 \right\}\\
	& & \ldots \\ 
	& \cup & \left\{\hat{b}^n_{i,j} =  b^n_i(s)b^n_j(t): p+2^{n-1}+1\leq i\leq p+2^{n} \mbox{ and } 1\leq j\leq p+2^{n} \right\},
\end{array}
\end{equation*}
where $B^i_{j}(t)$ is the $j$-th Bernstein polynomial of degree $i$. This defines a function space $\widehat{\mathcal{W}}_n$ over $\Delta$ via 
\begin{equation*}
	\widehat{\mathcal{W}}_n = \mbox{span}\left\{ 
	\beta^n_{i,j} = \hat{b}^n_{i,j}\circ{\f u}^{-1} \; : \; \hat{b}^n_{i,j} \in \widehat{\mathbb{S}}_n 
	\right\}.
\end{equation*}
\end{definition}
\begin{lemma}\label{lem:hatW}
The space $\widehat{\mathcal{W}}_n$ and the corresponding basis fulfill the following properties.
\begin{itemize}
\item[(A)] The space $\widehat{\mathcal{W}}_n$ is a subspace of ${\mathcal{W}}_n \cap \mathscr{C}^{p}(\overline{\Delta})$. 
\item[(B)] For all $\delta\in\widehat{\rm{T}}_n$ the space fulfills $\mathbb{P}^p \subseteq \widehat{\mathcal{W}}_n \;|_\delta \subseteq \mathbb{Q}^{p}$, where $\mathbb{P}^p$ is the space of bivariate polynomials of total degree $\leq p$ and $\mathbb{Q}^{p}$ is the space of bivariate polynomials of maximum degree $\leq p$.
\item[(C)] The basis functions fulfill
\begin{equation*}
	\beta^n_{i,j} (u,v) = \beta^{n-1}_{i,j} (2u,2v)
\end{equation*}
for all $1 \leq i \leq 2^{n-1}$ and for all $j$.
\item[(D)] The basis forms a partition of unity.
\end{itemize}
\end{lemma}
We omit the simple proof of this lemma.

\subsection{Stable projection operator}
\label{subsec:operator}

Let $n_0 = \left\lceil \log_2(8p) \right\rceil$. Then we have ${h_0} = 1/2^{n_0}$. In the following we consider $n\geq n_0$. This is necessary for the proofs of Lemma \ref{lem:regular-part} and Theorem \ref{thm:L2-stability}, to be able to properly split the contributions from the singular point at the left and the regular part at the right of the parameter domain. 
We consider the standard dual basis for univariate B-splines, see e.g. \cite{Schumaker2007,beirao2014actanumerica} for more details.
\begin{definition}\label{standard-dual-basis}
Let $\lambda^{n}_k : L^2(0,1)\rightarrow\RR$ be the dual basis for $b^n_i$ with $\lambda^{n}_k (b^n_i) = \delta^{k}_{i}$, for $i,k\in \{1,\ldots,2^n+p\}$.
\end{definition}
Moreover, consider a dual basis for polynomials on $\Delta_{h_0}$.
\begin{definition}\label{n0-dual-basis}
Let $\mu^{n_0}_{k,l} : L^2(\Delta_{h_0})\rightarrow\RR$ be dual functionals, with 
\begin{equation*}
	\mu^{n_0}_{k,l}( \beta^{n_0}_{i,j} ) = \delta^{k}_{i}\delta^{l}_{j},
\end{equation*}
for $k=1,\ldots,p+1$ and $l=1,\ldots,k$. 
\end{definition}
We assume that all the functionals in Definitions \ref{standard-dual-basis} and \ref{n0-dual-basis} are bounded in $L^2$. This is no restriction, see 
\cite{Schumaker2007}.

Let $c\f{I}$ be the mapping from $\RR^2 \rightarrow \RR^2$ with $(u,v)\mapsto (c\,u,c\,v)$.
Then we can define a global dual basis for all $n>n_0$. 
\begin{definition}
Let $n>n_0$. Given $\varphi \in L^2(\Delta)$, we set 
\begin{equation*}
	\Lambda^{n}_{k,l} (\varphi) = \mu^{n_0}_{k,l} (\varphi \circ (2^{n_0-n})\f{I})
\end{equation*}
for all $k=1,\ldots,p+1$ and all $l=1,\ldots,k$. Moreover, we set 
\begin{equation*}
	\Lambda^{n}_{k,l} (\varphi) = \lambda^{n}_{k}\otimes\lambda^{m(k)}_l(\varphi \circ \f u)
\end{equation*}
for all $k=p+2,\ldots,2^n+p$ and all $l=1,\ldots,2^{m(k)}+p$. Here $m(k) = \lceil \log_2(k-p+1) \rceil$. 
\end{definition}
It is easy to see, that $\Lambda^{n}_{k,l}$ is in fact a dual basis for the basis $\beta^{n}_{i,j}$ on $\Delta$.
 
\begin{definition}
The dual functionals naturally define a projection operator from $L^2(\Delta)$ onto $\widehat{\mathcal{W}}_n$, with
\begin{equation*}
	\Pi_{\widehat{\mathcal{W}}_n} (\varphi) = \sum^{2^n+p}_{k=1}\sum_{l} \Lambda^{n}_{k,l}(\varphi) \beta^{n}_{k,l},
\end{equation*}
where we appropriately sum over all $l$.
\end{definition}
Before we prove our main Theorem \ref{thm:L2-stability}, we recall two preliminary results.
\begin{lemma}\label{lem:regular-part}
Let $n\geq n_0$. There exists a constant $C_R>0$, depending only on $p$, such that for all $\varphi \in L^2(\Delta\backslash\Delta_{h_0})$ and for all $\delta \in \widehat{\rm{T}}_n \cap (\Delta \backslash \Delta_{3/8})$ we have 
\begin{equation}
	\left\| \Pi_{\widehat{\mathcal{W}}_n} \varphi \right\|_{L^2(\delta)} \leq C_R \left\| \varphi \right\|_{L^2(\tilde\delta)},
\end{equation}
where $\tilde\delta$ is the support extension of $\delta$ as presented in \cite{Bazilevs2006}. 
\end{lemma}
\begin{proof}
A proof of this lemma follows directly from the standard theory for regularly parameterized domains (see e.g. \cite{Bazilevs2006,beirao2014actanumerica}), since $\Delta\backslash\Delta_{h_0}$ is parameterized via $\f u_0: \left] h_0,1 \right[ \times \left] 0,1 \right[ \rightarrow \Delta\backslash\Delta_{h_0}$ and the support extension of any $\delta \in \Delta \backslash \Delta_{3/8}$ remains in $\Delta\backslash\Delta_{(3/8 - ph)}$. Since $n_0$ is chosen large enough, we have $\frac{3}{8}-ph \geq h_0$ for all $h\leq h_0$, and therefore $\tilde\delta \subset \Delta\backslash\Delta_{h_0}$. In that case, the constant $C_R$ depends on the degree $p$ and on the Jacobian determinant of the mapping $\f u_0$, which depends on $n_0$ only. Since $n_0$ is fully determined by $p$, the constant $C_R$ depends on $p$ only.
\end{proof}
On the refinement level $n_0$ it is obvious that the operator is bounded.
\begin{lemma}\label{lem:initial-level}
There exists a constant $C_0>0$, depending only on $p$, such that for all $\varphi \in L^2(\Delta)$ and for all $\delta \in \widehat{\rm{T}}_{n_0}$ we have 
\begin{equation}
	\left\| \Pi_{\widehat{\mathcal{W}}_{n_0}} \varphi \right\|_{L^2(\delta)} \leq C_0 \left\| \varphi \right\|_{L^2(\tilde\delta)}.
\end{equation}
\end{lemma}
\begin{proof}
This lemma follows directly from the boundedness of the dual functionals. There is no condition on the scaling needed, since we consider the coarse level $n_0$ only. The constant $C_0$ depends on $p$ and $n_0$ only. Again, $n_0$ is fully determined by $p$, which concludes the proof.
\end{proof}

Now we can prove the main result, the stability of the projection operator.
\begin{theorem}\label{thm:L2-stability}
Let $n\geq n_0$. There exists a constant $C_S>0$, depending only on $p$, such that for all $\varphi \in L^2(\Delta)$ and for all $\delta \in \widehat{\rm{T}}_n$ we have 
\begin{equation}
	\left\| \Pi_{\widehat{\mathcal{W}}_n} \varphi \right\|_{L^2(\delta)} \leq C_S \left\| \varphi \right\|_{L^2(\tilde\delta)}.
\end{equation}
\end{theorem}
\begin{proof}
We split the proof into three parts.
\begin{itemize}
\item[(A)] Let $\delta \subset \Delta \backslash \Delta_{3/8}$. Then this statement follows directly from Lemma \ref{lem:regular-part}.
\item[(B)] Let $\delta \subset \Delta_{3/8}$ and there exists an $m \in \mathbb{Z}^+$, with $m\leq n-n_0$, such that $2^m\f I (\delta) \subset \Delta_{3/4} \backslash \Delta_{3/8}$. Due to Lemma \ref{lem:hatW} (C) we have 
$$\widehat{\mathcal{W}}_n |_{\Delta_{3/8}} = \widehat{\mathcal{W}}_{n-1} |_{\Delta_{3/4}} \circ 2\f I$$
for all $n \geq  \log_2(8 p)$. Hence we conclude 
\begin{equation*}
	\left\| \Pi_{\widehat{\mathcal{W}}_n} \varphi \right\|^2_{L^2(\delta)} 
	= 2^{-m}\left\| \Pi_{\widehat{\mathcal{W}}_{n-m}} (\varphi\circ (2^{-m}\f I)) \right\|^2_{L^2(2^m\f I (\delta))}.
\end{equation*}
The desired result follows from Lemma \ref{lem:regular-part} and because 
\begin{equation*}
	2^{-m} \left\|\varphi\circ ((2^{-m})\f I) \right\|^2_{L^2(\widetilde{2^m\f I (\delta)})} = \left\| \varphi \right\|^2_{L^2(\tilde\delta)}.
\end{equation*}
\item[(C)] We have 
\begin{equation*}
	\left\| \Pi_{\widehat{\mathcal{W}}_n} \varphi \right\|^2_{L^2(\delta)} 
	= 2^{n_0-n}\left\| \Pi_{\widehat{\mathcal{W}}_{n_0}} (\varphi\circ (2^{n_0-n}\f I)) \right\|^2_{L^2(2^{n-n_0}\f I (\delta))}.
\end{equation*}
Here, the bound is fulfilled because of Lemma \ref{lem:initial-level}. 
\end{itemize}
This concludes the proof with $C_S = \max(C_R,C_0)$.
\end{proof}
Having a stable projection operator, we can prove approximation error bounds on $\Omega$ for a special class of geometry mappings $\f G$. But before we present the proof, we discuss a possible generalization. 

\subsection{Generalization to locally quasi-uniform knot vectors}
\label{subsec:generalization-quasi-uniform}

The results can be generalized to certain configurations of non-uniform refinements. 
Let ${\rm{T}}_n$ be a sequence of locally quasi-uniform meshes with ${\rm{T}}_0 = \{ \Delta \}$. 
In that case one can also show that there exists a bounded operator, projecting onto the piecewise polynomial function space over the locally quasi-uniform mesh. The stability constant then depends on the degree $p$ and on the quasi-uniformity constants, i.e. the maximum ratio between the diameter of an element and the diameter of the inscribed circle, as well as the maximum ratio between the diameters of neighboring elements. Figure \ref{fig:qu-meshes} depicts a sequence of three locally quasi-uniform meshes.
\begin{figure}[!ht]
    \centering
    \includegraphics[width=0.2\textwidth]{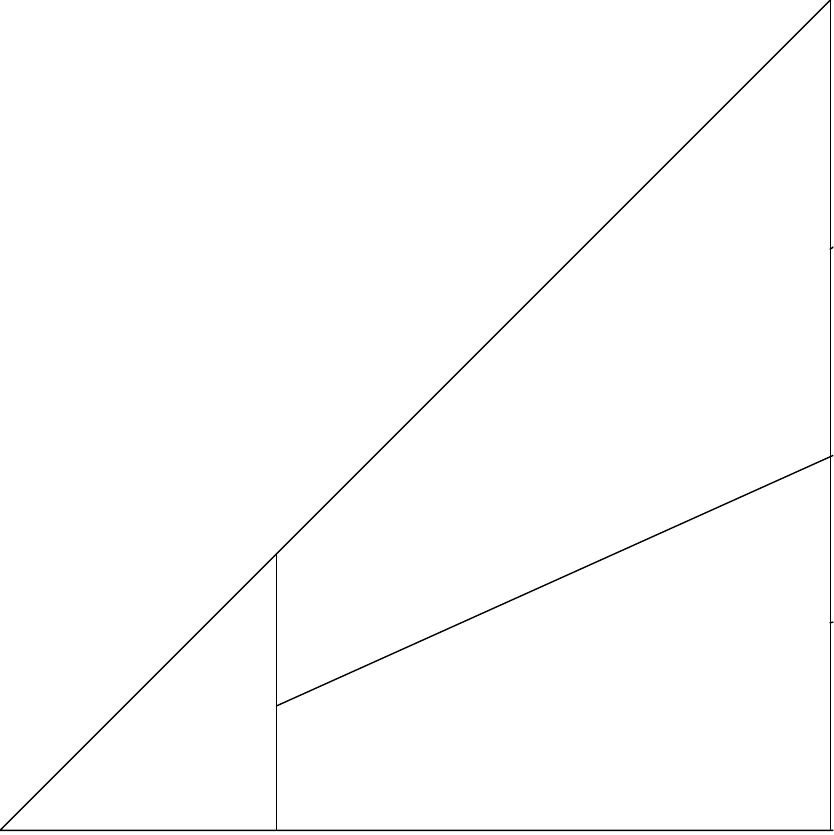}\qquad
    \includegraphics[width=0.2\textwidth]{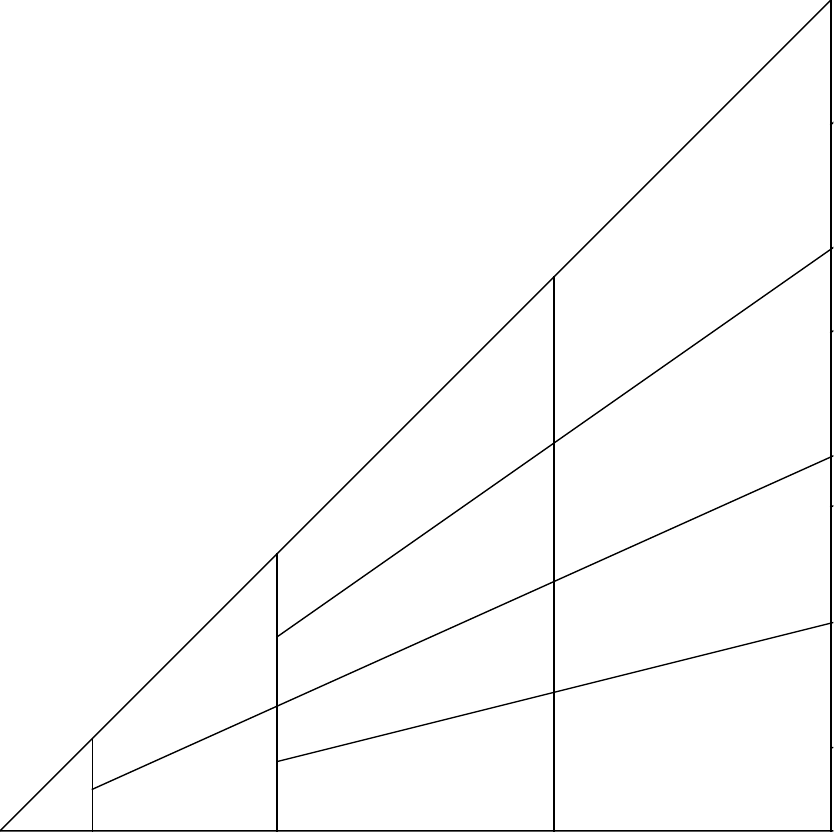}\qquad
    \includegraphics[width=0.2\textwidth]{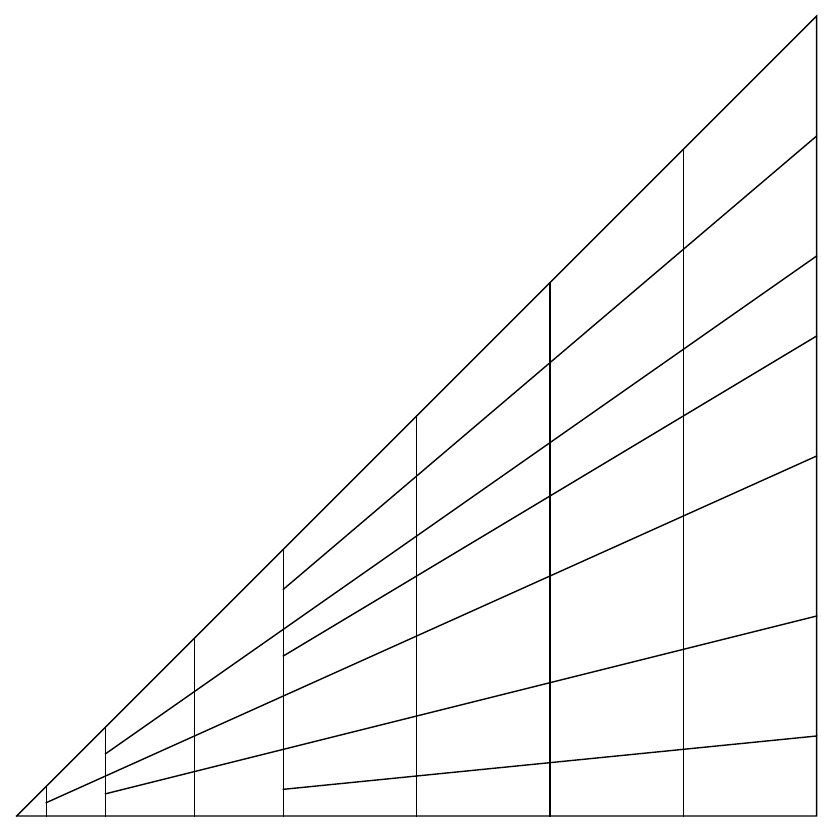}
    \caption{Meshes $\widehat{\rm{T}}_1$, $\widehat{\rm{T}}_2$ and $\widehat{\rm{T}}_3$ for a locally quasi-uniform refinement}\label{fig:qu-meshes}
\end{figure}

This generalization is based on the fact, that Lemma \ref{lem:regular-part} is actually valid for any mesh on $\Delta \backslash \Delta_{3/8}$ (see \cite{beirao2014actanumerica}). Moreover, Lemma \ref{lem:initial-level} can be generalized to arbitrary initial meshes. Then the constant depends only on the degree and on the quasi-uniformity of the initial mesh. We do not go into the details of this generalization here, but continue with the proof of approximation properties in the uniform case.

\section{Approximation error bounds}
\label{sec:appr-estimates}

We first prove approximation error bounds on the triangle $\Delta$ and then show bounds on $\Omega$, using the equivalence of norms presented in Assumption \ref{assu:equiv-norms}.

\subsection{Approximation on $\Delta$}

On $\Delta$ we have the following approximation error bound.
\begin{theorem}\label{thm:approx-delta}
Let $n \geq n_0$. Then there exists a constant $C>0$, depending only on the degree $p$, such that for all $\varphi \in \mathcal{H}^{p+1}(\Delta)$, for all $q\leq p+1$, and for all $\delta \in \widehat{T}_n$ we have 
\begin{equation*}
	\left| \varphi - \Pi_{\widehat{\mathcal{W}}_n} \varphi \right|_{H^q(\delta)} \leq C h^{p-q+1} |\varphi |_{\mathcal{H}^{p+1}(\tilde\delta)}.
\end{equation*}
\end{theorem}
\begin{proof}
We show this result only for $q=0$. The statement for general $q$ follows immediately. 
According to \cite{BrambleHilbert1970}, there exists a constant $C_A$, depending only on $p$, and a polynomial $\psi \in \PP^p$, such that 
\begin{equation*}
	\left\| \varphi - \psi \right\|_{L^2(\tilde\delta)} \leq C_A h^{p+1} |\varphi |_{\mathcal{H}^{p+1}(\tilde\delta)}.
\end{equation*}
The constant is uniformly bounded, due to the quasi-uniformity of the support extension $\tilde\delta$ of the element $\delta$.
Due to Lemma \ref{lem:hatW} (B) we have $\psi \in \widehat{\mathcal{W}}_n$, hence 
\begin{equation*}
	\left\| \varphi - \Pi_{\widehat{\mathcal{W}}_n} \varphi \right\|^2_{L^2(\delta)}
	\leq \left\| \varphi - \psi \right\|^2_{L^2(\delta)} + \left\| \Pi_{\widehat{\mathcal{W}}_n} ( \varphi - \psi ) \right\|^2_{L^2(\delta)}.
\end{equation*}
From Theorem \ref{thm:L2-stability} it follows that 
\begin{equation*}
	\left\| \Pi_{\widehat{\mathcal{W}}_n} ( \varphi - \psi ) \right\|^2_{L^2(\delta)} \leq C_S^2 \left\| \varphi - \psi \right\|^2_{L^2(\tilde\delta)}.
\end{equation*}
Hence, we conclude 
\begin{equation*}
	\left\| \varphi - \Pi_{\widehat{\mathcal{W}}_n} \varphi \right\|^2_{L^2(\delta)} \leq (1+C_S^2) \left\| \varphi - \psi \right\|^2_{L^2(\tilde\delta)} \leq (1+C_S^2) C_A^2  h^{2p+2} |\varphi |^2_{\mathcal{H}^{p+1}(\tilde\delta)},
\end{equation*}
which concludes the proof with $C = C_A \sqrt{1+C_S^2}$.
\end{proof}
We can now apply this theorem to prove bounds on the mapped physical domain $\Omega$.

\subsection{Approximation on $\Omega$}

This bound can be extended directly to some geometry mapping $\f G$ fulfilling additional regularity criteria.
\begin{theorem}
Let $n\geq \max(n_0,n^*)$, let $\f G = \f F \circ \f u$, where $\f F = \frac{1}{F_0}(F_1,F_2)^T$ with $F_0,F_1,F_2 \in \widehat{\mathcal{W}}_{n^*}$ and $\det\nabla\f F > \underline c > 0$, and let 
\begin{equation*}
	\Pi_{{\mathcal{V}}_n} (\varphi) = \frac{\Pi_{\widehat{\mathcal{W}}_n} ((\varphi \circ \f F).F_0)}{F_0}\circ \f F^{-1}.
\end{equation*}
Then there exists a constant $C>0$, depending only on the degree $p$ and on the geometry parameterization $\f G$, such that for all $\varphi \in {H}^{p+1}(\Omega)$, for all $q\leq p+1$, and for all $\omega = \f F(\delta)$, with $\delta \in \widehat{T}_n$, we have 
\begin{equation*}
	\left| \varphi - \Pi_{{\mathcal{V}}_n} \varphi \right|_{H^q(\omega)} \leq C h^{p-q+1} \|\varphi \|_{{H}^{p+1}(\tilde\omega)}.
\end{equation*}
\end{theorem}
\begin{proof}
The proof of this theorem is a direct consequence of Theorem \ref{thm:approx-delta} and the equivalence of norms as stated in Assumption \ref{assu:equiv-norms}. The assumption is fulfilled because of Lemma \ref{lem:hatW} (A) and Proposition \ref{prop:equiv-norms}.
\end{proof}
Note that the condition on the geometry mapping $\f G$ is a real restriction. A general NURBS mapping $\f G$ cannot be represented as a composition of a singular bilinear mapping $\f u$ with a regular mapping $\f F= \frac{1}{F_0}(F_1,F_2)^T$ which fulfills $F_0,F_1,F_2 \in \widehat{\mathcal{W}}_n$. Hence, for general NURBS parameterizations, the equivalence of norms in Assumption \ref{assu:equiv-norms} is not fulfilled. However, numerical evidence shows, that in many cases the equivalence of norms is not necessary for optimal order of approximation. This phenomenon may be studied in more detail. 

\section{Conclusion}
\label{sec:conclusion}

In this paper we could prove approximation error bounds for isogeometric function spaces over singularly parameterized domains, thus extending known results in isogeometric analysis to singular patches. We show these bounds for a hierarchically refined subspace of the full tensor product spline space. The refinement is considered to be uniform with $h$, which is not a necessary condition, as we pointed out in Section \ref{subsec:generalization-quasi-uniform}. We however restricted our study to a certain type of singular patches that are derived from regularly mapped triangular domains. In the framework we presented, this is necessary in order to guarantee the equivalence of norms between the triangle $\Delta$ and the mapped domain $\Omega$. It is not clear yet, how these results extend to more general configurations where this equivalence is not satisfied near the singularity.

\end{document}